\documentclass[11pt]{article}

\usepackage{amsfonts}
\usepackage{bezier}
\usepackage{euscript}
\usepackage{amsfonts,amssymb,amsmath,amsbsy,amsthm,amscd}
\usepackage[T2A]{fontenc}
\usepackage[french, german, english]{babel}   

\newcommand{\cH}{\mathcal{H}}
\newcommand{\cM}{\mathcal{M}}
\newcommand{\cR}{\mathcal{R}}

\newcommand{\cP}{\mathcal{P}}
\newcommand{\cB}{\mathcal{B}}

\newcommand{\cGH}{\mathcal{GH}}

\newcommand{\dis}{{\operatorname{dis}}\,}
\newcommand{\diam}{{\operatorname{diam}}}
\renewcommand{\min}{{\operatorname{min}}}

\newcommand{\sm}{\setminus}
\newcommand{\bl}{\bigl}
\newcommand{\br}{\bigr}

\theoremstyle{plain}
\newtheorem{thm}{Theorem}
\newtheorem{prop}{Proposition}
\newtheorem{cor}{Corollary}

\theoremstyle{definition}

\renewenvironment{proof}{{\bfseries Proof.}}{}

\begin{document}

\date{}
\title{Metric Segments in Gromov--Hausdorff class.}
\author{\large{Borisova Olga Borisovna}\\
\small{Student, Lomonosov State University, Faculty of Mechanics and Mathematics}
}
\maketitle
\begin{abstract}
We study properties of metric segments in the class of all metric spaces considered up to an isometry, endowed with Gromov--Hausdorff distance. On the isometry classes of all compact metric spaces, the Gromov-Hausdorff distance is a metric. A metric segment is a class that consists of points lying between two given ones. By von Neumann--Bernays--\selectlanguage{german}{G"odel} (NBG) axiomatic set theory, a proper class is a ``monster collection'', e.g., the collection of all cardinal sets. We prove that any metric segment in the proper class of isometry classes of all metric spaces with the Gromov-Hausdorff distance is a proper class if the segment contains at least one metric space at positive distances from the segment endpoints. In addition, we show that the restriction of a non-degenerated metric segment to compact metric spaces is a non-compact set.

\begin{keywords} %
    Gromov--Hausdorff distance, class of all metric spaces, von Neumann--Bernays--\selectlanguage{german}{G"odel} axioms, metric
segment
\end{keywords}
\end{abstract}

\section{Introduction}
\markright{\thesection.~Introduction}
The Gromov-Hausdorff distance is a measure of difference between two arbitrary metric spaces. It is closely related with Hausdorff distance. The Hausdorff distance first appeared in the book of Hausdorff entitled ``Set theory''  \cite{1} in 1914. This value defines a natural distance on non-empty subsets of a metric space and generates a metric on the set of all non-empty closed bounded subsets.

In 1981, Gromov in his monograph \cite{2} introduced a distance between arbitrary metric spaces: he embedded isometrically two metric spaces into common metric spaces and considered the Hausdorff distances between their images. The infimum of the possible Hausdorff distances over all such embeddings is called the Gromov-Hausdorff distance.


It is well-known that this distance satisfies the triangle inequality and vanishes for any isometric metric spaces. It is correctly to consider the Gromov-Hausdorff distance on isometry classes of metric spaces, which can be calculated on arbitrary  representatives of the classes (it does not depend on the choice of the representatives). Since any non-empty set can be endowed with some metric (for example, one can put all non-zero distances to be equal to 1), all isometry classes do not form a set due to the famous Cantor's paradox. In order to overcome this problem, we use in this paper the von Neumann--Bernays--\selectlanguage{german}{G"odel} (NBG) axiomatic set theory \cite{2a}.


In NBG, all the objects (analogues of usual sets) are referred as classes. A class is a set if there exists a class containing this one as an element, otherwise, the class is called a proper class. The class of all sets is a proper class. For the classes many standard operations are
defined, for example, products, mappings, etc. Therefore, on proper classes, similarly with the case of sets, a distance is defined correctly. The class of isometry classes of all metric spaces (as mentioned above) is a proper class, and being endowed with the Gromov-Hausdorff distance, it is an extended pseudometric space, which we denote by $\cGH$. 

In this paper, we investigate the class of elements lying between two given points in extended pdeudometric space $\cGH$. A metric space $Z\in \cGH$ lies between $X,Y \in \cGH$ if $|XZ|+|ZY|=|XY|$, where $|\cdot \cdot |$ denotes the Gromov-Hausdorff distance. The class of all such $Z$ is called a metric segment and is denoted by $[X,Y]$. We prove constructively, that if a metric segment contains $Z\in \cGH$ at positive distances from endpoints, then this space $Z$ can be modified by adding a set with an arbitrary cardinal number such that the new metric space remains in the segment. Therefore, such a metric segment in $\cGH$ is a proper class.

The Gromov-Hausdorff distance is well studied on the set of all compact metric spaces considered up to isometry, where it is a metric. The set of all isometry classes of compact metric spaces is called the \emph{Gromov-Hausdorff space}. In 2015, A. Ivanov, N. Nikolaeva, A. Tuzhilin showed that in Gromov-Hausdorff space any two points are joined by geodesics \cite{nik}.
Therefore, any non-degenerated segment is a non-empty set.
The previous result leads to non-compactness of the metric segments in the Gromov--Hausdorff space.

I would like to express my gratitude to my advisor Dr. Sc. (Phys.-Math.) professor Tuzhilin A.A. and also to Dr. Sc. (Phys.-Math.) professor Ivanov A.O. for permanent support of my study.

The work is partly supported by RFBR (Project \textnumero 19-01-00775р).

\section{Preliminaries}
\markright{\thesection.~Preliminaries}

Let $X$ be an arbitrary set. A \emph{distance function on $X$} is any symmetric mapping $d\colon X \times X \to [0,\infty]$ vanishing at the pairs of coinciding elements. If $d$ satisfies the triangle inequality, then the mapping $d$ is called an \emph{extended pseudometric}. If in addition $d(x, y) > 0$ for all $x \neq y$, the $d$ is referred as an \emph{extended metric}.
If also $d(x, y) < \infty$ for all $x,y \in X$, then such distance function is called a $metric$, and sometimes a \emph{finite metric}
to highlight the difference from an extended metric. A set $X$ with an (extended) (pseudo-) metric is called an (\emph{extended}) (\emph{pseudo-})\emph{metric space}.
 
Let $X$ be an arbitrary metric space. The distance between points $x, y \in X$ we denote by $|xy|$. Sometimes, we add subscript $|xy|_X$ to highlight the location of the points. For any $x\in X$ and non-empty $A \subset X$, we put $|xA| = \inf \bigl\{ |xa| : a\in A\bigr\}$. Let $\cP(X)$ stand for the set of all non-empty subsets of $X$. For any $A, B \in \cP(X)$ let us put
$$
d_H(A, B) = \max\bigl\{\sup\limits_{a\in A} |aB|,\sup \limits_{b\in B}|Ab|\bigr\}.
$$
The value $d_H(A, B)$ is called \emph{the Hausdorff distance between $A$ and $B$}.

Let $\mathcal B(X) \subset \cP(X)$ be the family of all non-empty bounded subsets of $X$.

\begin{prop}[\cite{3}]\label{prop0}
The function $d_H(A, B)$ is a pseudometric on $\mathcal B(X)$.
\end{prop}

 By  $\cH(X) \subset \mathcal B(X)$ we denote the family of all non-empty closed bounded subsets of a metric space $X$.
\begin{prop}[\cite{3}]\label{prop1}
The function $d_H(A, B)$ is a metric on $\mathcal H(X)$.
\end{prop}

 Let $X$ and $Y$ be metric spaces. A triple $(X', Y', Z)$ consisting of a metric space $Z$ and its subsets $X'$ and $Y'$ isometric to $X$ and $Y$, respectively, is called a \emph{realization of the pair $(X,Y)$}. For any metric spaces $X, Y$, the value
$$
d_{GH}(X,Y) = \inf\bigl\{r : \exists \, (X', Y', Z),\ d_H(X', Y') \leq r\bigr\}
$$
 is called \emph{the Gromov--Hausdorff distance between $X$ and $Y$}.

\begin{prop}[\cite{3}]
The function $d_{GH}(X, Y)$ satisfies the triangle inequality. If $X$ and $Y$ are isometric metric spaces, then $d_{GH}(A,B) = 0$.
\end{prop}
 Notice, that the Gromov--Hausdorff distance could take infinite values, and there are some examples of non-isometric metric spaces $X, Y$ for which $d_{GH}(X,Y) = 0$, see \cite{3}. So, the Gromov--Hausdorff distance is well-defined  on  isometry classes of metric spaces (it does not depend on the choice of representatives in the classes).

Since any non-empty set can be endowed with some metric (for example, one can put all non-zero distances to be equal to 1), so there are ``as many'' isometry classes of metric spaces as all possible sets, i.e., the family of isometry classes is not a set, but it is a proper class which, together with the Gromov--Hausdorff distance, is denoted by $\cGH$. Here we use the concept of class in the sense of
von Neumann--Bernays--G"odel Set Theory (NBG). Recall some concepts of the NBG.

In NBG, all the objects (analogues of usual sets) are referred as \emph{classes}. There are two types of classes: the \emph{sets} and the \emph{proper classes}.
 A class $\mathcal A$ is called a \emph{set} if there exists a class $\mathcal C$ such that $\mathcal A \in C$.
 A class $\mathcal A$ is called a \emph{proper class} if for any class $\mathcal C$ it holds $\mathcal A \not\in C$.
\begin{prop}[\cite{2a}]
The class of all sets $\mathcal V = \{ \mathcal A:\mathcal A = \mathcal A\}$ is a proper class.
\end{prop}

For all classes $\mathcal X, \mathcal Y$ it is defined $\mathcal X \times \mathcal X$, $f\colon \mathcal X \to \mathcal Y$, etc., in particular, we can speak about distance function on a class.
Therefore, the distance on $\cGH$ is defined correctly.
The class of isometry classes of bounded metric spaces is a proper class too, and being endowed with the Gromov-Hausdorff distance, it is denoted by $\cB$.

\begin{prop}[\cite{3}]\label{prop2a}
The function $d_{GH}(A, B)$ is a pseudometric on $\cB$.
\end{prop}

The most well-investigated subset of $\cGH$ is the set $\cM$ of isometry classes of compact metric spaces.

\begin{prop}[\cite{3}]\label{prop2}
The function $d_{GH}(A, B)$ is a metric on $\cM$.
\end{prop}

Now we describe a convenient way to calculate the Gromov--Hausdorff distance. Let $X$ and $Y$ be non-empty sets. Put $\cP(X, Y ) = \cP(X\times Y )$. Each element of $\cP(X, Y )$ is called  a \emph{relation between $X$ and $Y$}.

Let $\pi_X \colon X\times Y \to X$ and $\pi_Y \colon X\times Y \to Y$ be the canonical projections $\pi_X(x, y) = x$ and $\pi_Y (x, y) = y$. A relation $\sigma \in \cP(X, Y )$
is called a \emph{correspondence}, if the restrictions of $\pi_X$ and $\pi_Y$ onto $\sigma$ are surjections. The set of all correspondences between $X$ and $Y$ is denoted by $\cR(X, Y )$.

For any metric spaces $X$ and $Y$ and for any relation $\sigma \in \cP(X, Y )$, we define its \emph{distortion} as follows:

$$
\dis \sigma = \sup\Bigl\{\bigl||xx^\prime|-|yy^\prime|\bigr|:(x, y),\ (x^\prime, y^\prime)\in \sigma\Bigr\}.
$$

\begin{prop}[\cite{3}]\label{prop3}
Let $X$ and $Y$ be metric spaces. Then
$$
d_{GH}(X, Y)= \frac 1 2 \inf \bigl\{ \dis R:R\in \cR(X,Y)\bigr\}.
$$
\end{prop}

A correspondence $R \in \cR(X,Y)$ is called \emph{optimal} if $d_{GH}(X,Y) = \frac 1 2\dis R$. The set of all optimal correspondences between $X$ and $Y$ we denote by $\cR_{\operatorname{opt}}(X,Y)$. Let $\cR_c(X, Y)$ be the set of all correspondences closed in $X\times Y$.

\begin{prop}[\cite{5}]\label{prop4}
For $X, Y \in \cM$ it holds $\cR_{\operatorname{opt}}(X, Y)\cap \cR_c(X, Y) \neq \emptyset$.
\end{prop}

Let $X, Y$ be arbitrary metric spaces. For each $t \in (0, 1)$ we define a distance function on $X \times Y$ as
$\bl|(x, y)(x^\prime, y^\prime)\bigr|_t = (1 - t)|xx^\prime| + t|yy^\prime|$.

\begin{prop}[\cite{5}]\label{prop5}
The function defined above $|\cdot|_t$ is a metric for any $t \in (0, 1)$ on $X \times Y$.
\end{prop}

For any $\sigma \in \cP(X, Y )$, the metric space $\bigl(\sigma, |\cdot|_t\bigr)$ is denoted by $\sigma_t$.

\begin{prop}[\cite{5}]\label{prop6}
For any $X, Y \in \cM$ and any closed $\sigma \in \cP(X, Y )$, the metric space
$\sigma_t$ belongs to $\cM$ for all $t \in (0, 1)$.
\end{prop}

Let us choose an arbitrary $R \in \cR(X, Y )$ and extend the family $R_t, t \in (0, 1)$, up to $t = 0, 1$,
as $R_0 = X$ and $R_1 = Y$.

\begin{prop}[\cite{5}]\label{prop7}
For any $X,Y \in \cM$ and $R \in \cR_{opt}(X, Y ) \cap \cR_c(X, Y )$, the mapping $t \mapsto R_t, t \in [0, 1]$, is a curve connecting $X$ and $Y$. This curve is a shortest curve, and its length equals $d_{GH}(X, Y )$. Therefore, the metric of the space $\cM$ is geodesic.
\end{prop}
Let $X$ be a metric space and $\varepsilon > 0$. \emph{A covering number} $\operatorname{cov}(X,\varepsilon)$ is the minimum number of opened balls with radius $\varepsilon$ one needs in order to cover $X$.

The following proposition is called the \emph{precompact theorem of Gromov}.

For metric space $X$ we denote the diameter of space as $\diam X$.

\begin{prop}[\cite{3}]\label{thmG}
Let $C$ be a non-empty subset $\cM$. Then the following statements are equivalent.
\begin{enumerate}
\item There exists a number $D \geq 0$ and a function $N \colon (0, \infty)   \rightarrow \mathbb N$ such that for any $X \in C$ the inequalities $\diam X \leq D$ and $\operatorname{cov}(X, \varepsilon) \leq N(\varepsilon)$ are valid.
\item The family $C \subset \cM$ is precompact.
\end{enumerate}
\end{prop}

A class $X$ we call \emph{a simplex} if all its non-zero distances are equal. For any cardinal $n$, we denote by $\lambda \Delta_n$ the simplex with $n$ vertices and with distance $\lambda$ between any its distinct points. For $\lambda = 1$ the space $\lambda \Delta_n$ we denote by $\Delta_n$ for short.
\section{Main results}
\markright{\thesection.~ Main results}
If $X$ is a class with some distance function, then a \emph{metric segment} with ends $A, B\in X$ is a class $[A, B] = \bigr\{C \in X : |AC| + |CB| = |AB| \bigl\}$. If $|AB| > 0$, then the segment $[A,B]$ is called \emph{non-degenerate}.

A closed ball in $X$ with center $x_0\in X$ and radius $r > 0$ we denote by $B_r(x_0) = \bl\{x\in X : |xx_0|\leq r\br\}$.
\begin{thm}\label{mythm1}
Let a segment $[X,Y]$ be non-degenerate for $X, Y \in \cGH$. If there exists a metric space $Z\in \cGH$ lying in $[X,Y]$ such that $d_{GH}(X, Z) > 0$ and $d_{GH}(Z, Y) > 0$, then the segment $[X,Y]$ contains a space $Z^*\in \cGH$ with at least one isolated point, and the conditions $d_{GH}(X, Z^*) > 0$ and $d_{GH}(Z^*, Y) > 0$ hold.
\end{thm}
\begin{proof}
Suppose that $Z$ does not contain isolated points, otherwise the theorem is proved. Fix an arbitrary point $z_0 \in Z$ and a finite number $\delta$ such that
$$
0 < \delta \leq \min\bl\{2d_{GH}(X, Z), 2d_{GH}(Y,Z)\br\}.
$$
We construct a metric space $Z^* = Z \cup\{z^*\}$ preserving the distances between points $z\in Z$ and assigning the distances from $z\in Z^*$ to $z^*$ as follows:
$$
|z^*z|_{Z^*} =
\begin{cases}
    \delta &\text{for $z \in Z \cap B_{\delta}(z_0)$},\\
    |z_0z|_Z &\text{for $z \in Z \setminus B_{\delta}(z_0)$},\\
    0 &\text{$z = z^*$}.\\
 \end{cases}
$$

Make sure that this function is a metric. It is sufficient to check the triangle inequality for $z_1, z_2, z^*$, where $z_1, z_2 \in Z$, because other axioms are obvious.

Consider different cases of the locations of points $z_1, z_2$.

Let $z_1, z_2 \in Z\cap B_{\delta}(z_0)$ then the triangle is isosceles and
$$
|z^*z_1|_{Z^*} \leq |z^*z_2|_{Z^*} + |z_2z_1|_{Z^*},
$$
$$
|z_1z_2|_{Z^*} = |z_1z_2|_{Z} \leq |z_1z_0|_{Z} + |z_0z_2|_{Z} \leq 2\delta = |z_1z^*|_{Z^*} + |z^*z_2|_{Z^*}.
$$
If $z_1, z_2 \in Z \setminus B_{\delta}(z_0)$, then the distances between $z_1, z_2, z^*$ are equal to the distances in $Z$ between $z_1, z_2, z_0$, respectively.
In the last case if $z_1 \in Z \setminus B_{\delta}(z_0)$ and $z_2 \in Z \cap B_{\delta}(z_0)$, then
$$
|z^*z_1|_{Z^*} = |z_0z_1|_{Z} \leq |z_0z_2|_{Z} + |z_2z_1|_{Z} \leq |z^*z_2|_{Z^*} + |z_2z_1|_{Z^*},
$$
$$
|z^*z_2|_{Z^*} = \delta \leq |z^*z_1|_{Z^*} \leq |z^*z_1|_{Z^*} + |z_1z_2|_{Z^*},\\
$$
$$
|z_1z_2|_{Z^*} = |z_1z_2|_{Z} \leq |z_1z_0|_{Z} + |z_0z_2|_{Z} \leq |z_1z^*|_{Z^*} + |z^*z_2|_{Z^*}.
$$
Therefore $Z^*\in \cB$.

The next step is to prove $d_{GH}(X, Z^*) \leq d_{GH}(X, Z)$.

Consider an arbitrary correspondence $R \in \cR(X, Z)$ and construct a new correspondence $R^* = R \cup \bigl(R^{-1}(z_0)\times \{z^*\}\bigr) \in \cR(X, Z^*)$, where $R^{-1}(z_0) = \bl\{x\in X: (x, z_0)\in R\br\}$. The distortion of $R^*$ consists of three parts:
$$
    r_1 = \sup\Bigl\{\bigl||xx'|-|zz'|\bigr| :(x, z), (x', z')\in R^*; z, z' \in Z\Bigr\},
$$
$$
    r_2 = \sup\Bigl\{\bigl||xx'| - |zz^*|\bigr|: |zz_0| > \delta; (x, z) \in R^*; x'\in R^{-1}(z_0) \Bigr\},
$$
$$
    r_3 = \sup\Bigl\{\bigl||xx'| - |zz^*|\bigr|: |zz_0| \leq \delta; (x, z) \in R^*; x'\in R^{-1}(z_0) \Bigr\},
$$
i.e, $\dis R^* = \max\{r_1, r_2, r_3\}$. Estimate $r_i$ from above.

By construction of correspondence $R^*$, we have $r_1 = \dis R$.

For $r_2$, the inequality $|zz_0| > \delta$ holds, then $|zz^*| = |zz_0|$ and
$$
   r_2 = \sup\Bigl\{\bigl||xx'|-|zz_0|\bigr|: |zz_0| > \delta; (x, z), (x',z_0) \in R \Bigr\} \leq \dis R.
$$
Consider any points $x,x',z$ from the definition of $r_3$.
If $0 \leq |xx'|-|zz^*|$, then
$$
|xx'|-|zz^*|\leq|xx'|-|zz_0|\leq \dis R.
$$ Otherwise,
$$
0 < |zz^*|-|xx'|\leq \delta \leq 2d_{GH}(X, Z) \leq \dis R.$$
Hence, $r_3 \leq \dis R$, and we get $\dis R^* \leq \dis R$.

Using this inequality and Proposition \ref{prop3}, we obtain
$
d_{GH}(X, Z^*) \leq d_{GH}(X, Z).
$

The inequality $d_{GH}(Z, Y) \geq d_{GH}(Z^*, Y )$ can be proved similarly.
Then
$$
d_{GH}(X, Y ) = d_{GH}(X, Z) + d_{GH}(Z, Y ) \geq d_{GH}(X, Z^*) + d_{GH}(Z^*, Y ) \geq d_{GH}(X, Y).
$$
The last inequality is the triangle inequality for the Gromov--Hausdorff distance.
Thus, $d_{GH}(X, Z^*) + d_{GH}(Z^*, Y ) = d_{GH}(X, Y)$, which means $Z^* \in [X, Y]$.
\end{proof}

\begin{thm}\label{mythm2}
Let a segment $[X,Y]$ be non-degenerate for $X, Y \in \cGH$, and there is a metric space $Z\in [X,Y]$ with at least one isolated point such that $d_{GH}(X, Z) > 0$, $d_{GH}(Z, Y) > 0$. Then for an arbitrary cardinal $m$ there is a metric space $W(m)\in\cGH$ lying in the segment $[X,Y]$, where $W(m)$ contains a simplex with $m$ vertices. If $Z \in \cM$ and the cardinal $m$ is a finite, then $W(m) \in \cM$.

\end{thm}

\begin{proof}
Let $z^* \in Z$ be an isolated point.
Fix a finite number $\mu$ such that
$$
0 <\mu < 2\min\bigl\{d_{GH}(X, Z), d_{GH}(Z, Y ), S(z^*)\bigr\},
$$ where $S(z^*) = \inf \bigl\{|z^*z|: z\in Z, z \neq z^*\bigr\} > 0$. Let $m$ be an arbitrary cardinal number. On the set
$$W = W(\mu, m) = \mu\Delta_m \cup Z \setminus \{z^*\}$$ we define the distance function as follows.
 Let $w_1, w_2 \in W$, then
\begin{equation*}
|w_1w_2|_W =
 \begin{cases}
    |w_1w_2|_Z &\text{for $w_1, w_2 \in Z \backslash \{ z^*\},$}\\
    |w_iz^*|_Z &\text{for $w_i \in Z, w_{3-i} \in \mu\Delta_m, i\in \{1, 2\}$},\\
    \mu &\text{for $w_1, w_2 \in \mu\Delta_m, w_1 \neq w_2 ,$}\\
    0 &\text{otherwise.}
 \end{cases}
\end{equation*}

This function is a metric. The properties of positive definiteness and symmetricity are obvious. We need to check the triangle inequality for different points lying in $W$. If all three points $w_1, w_2, w_3$ lie in $Z \setminus \{z^*\}$, or if $w_i \in \mu\Delta_m$, where $i \in \{1, 2, 3\}$, then the triangle inequality holds . If $w_1, w_2 \in Z \backslash \{z^*\}$,
and $w_3 \in \mu\Delta_m$, then the distances between this points are equal to the distances between
$w_1, w_2, z^* \in Z$, respectively, so the triangle inequality is valid in this case. If $w_1 \in Z \backslash \{z^*\}$, and $w_2, w_3 \in \mu\Delta_m$,
then
$$
|w_2w_3|_W = \mu < 2S(z^*) \leq 2|z^*w_1|_Z = |w_2w_1|_W + |w_1w_3|_W,
$$
$$
|w_1w_2|_W = |w_1z^*|_Z < |w_1z^*|_Z + \mu = |w_1w_3|_W+|w_3w_2|_W.
$$
Similarly for $|w_1w_3|_W$.
Therefore, $W$ is a metric space.

Now we prove that  $d_{GH}(X, W) \leq d_{GH}(X, Z)$.

Consider an arbitrary correspondence $R \in \cR(X, Z)$ and construct a new correspondence $V \in \cR(X, W)$ as follows.
Put $(x, w) \in V$ if and only if $w \in Z \sm \{z^*\}$ and $(x, w) \in R$ or $w \in \mu\Delta_m$ and $(x, z^*) \in R$.

Calculate the distortion of $V$:
$$
\dis V = \sup\Bigl\{\bigl||xx^\prime| - |ww^\prime|\bigr|:(x, w), (x^\prime, w^\prime)\in V\Bigr\}.
$$
The distortion of $V$ consists of two parts:
$$
v_1 = \sup\Bigl\{\bigl||xx^\prime| - |ww^\prime|\bigr|:(x, w), (x^\prime, w^\prime)\in V; w\in Z\setminus\{z^*\}, w^\prime \in W\Bigr\},
$$
$$
v_2 = \sup\Bigl\{\bigl||xx^\prime| - |ww^\prime|\bigr|:(x, w), (x^\prime, w^\prime)\in V; w,  w^\prime \in \mu\Delta_m\Bigr\},
$$
i.e., $\dis V = \max\{v_1, v_2\}$.
For any pairs $(x, w), (x', w') \in V$ from the $v_1$, there are pairs $(x, z), (x', z') \in R$ with $z, z' \in Z$ and
$|ww'|_W = |zz'|_Z$, thus, $v_1 \leq \dis R$.
Estimate the value $v_2$ from above. Rewrite $v_2$ in equivalent form taking into account that $|ww'|_W = \mu$ for $w, w' \in \mu\Delta_m$, and therefore
$$
v_2 = \sup\Bigl\{\bigl||xx^\prime| - \mu\br|:(x, z^*), (x', z^*) \in R\Bigr\}.
$$
Due to the restriction on $\mu$, we have
$$
\mu < 2d_{GH}(X, Z) \leq \dis R.
$$
In addition, for $x, x' \in X$ such that $(x, z^*), (x', z^*) \in R$ we have $|xx'| \leq \dis R$.
Since the absolute value of difference between $\mu$ and $|xx'|$ is less or equal than $\dis R$, the value $v_2$ is less or equal than $\dis R$.

So we proved $\dis V = \max\{v_1, v_2\}\leq \dis R$.

According to this inequality and Proposition \ref{prop3}, we have $d_{GH}(X, W) \leq d_{GH}(X, Z)$.

Similarly, we obtain the inequality $d_{GH}(Z, Y) \geq d_{GH}(W, Y ).$

Then
$$
d_{GH}(X, Y ) = d_{GH}(X, Z) + d_{GH}(Z, Y ) \geq d_{GH}(X, W) + d_{GH}(W, Y ) \ge d_{GH}(X, Y),
$$
where the last inequality is the triangle inequality for the Gromov-Hausdorff distance.

Thus, $d_{GH}(X, W) + d_{GH}(W, Y) = d_{GH}(X, Y)$, which means $W(\mu, m) \in [X, Y]$ for any cardinal $m$.

If $Z$ is a compact metric space, and $m$ is a finite number, then the metric space $W(m)$ is a compact metric, because when constructing $W(m)$, we add only a finite number of points to compact metric space.




\end{proof}
\begin{cor}
 If for space $X, Y \in \cGH$ the metric segment $[X, Y]$  contains a metric space $Z \in \cGH$ such that $d_{GH}(X,Z) > 0$ and $d_{GH}(Y,Z) > 0$, then the metric segment $[X,Y]$ is a proper class.
\end{cor}
\begin{proof}
According to Theorem \ref{mythm1} and Theorem \ref{mythm2}, the metric segment $[X,Y]$ can be surjectively mapped to the proper class of all cardinality sets. For instance, any space $W(m)$ goes to a set with cardinal number $m$, and other elements of the segment go to single-point set. Therefore, the segment ``is not less'' than a proper class, so the segment is a proper class.
\end{proof}

\begin{cor}
For two different metric space $X,Y \in \cM$, the segment $[X,Y]$ is not compact.
\end{cor}
\begin{proof}
By Proposition \ref{prop7}, the set $[X,Y ] \setminus \{X,Y \}$ is non-empty, thus it contains some $Z$. Space $\cM$ is a metric space thus $d_{GH}(X,Z) > 0$ and $d_{GH}(Z,Y)>0$. According to Theorem \ref{mythm1} and Theorem \ref{mythm2} there exists a compact metric space $W(m)\in [X,Y]$ for any positive integer $m$ and $W(m)$ contains an $m$-vertices simplex. Let the distance between differen points in the simplex be $\mu > 0$.
For any $0< \varepsilon < \frac \mu 2$ the covering number $\operatorname{cov}\bl(W(m), \varepsilon\br)$ is not less than $m$, because each ball with radius $\varepsilon$ can cover no more than one point from $\mu\Delta_m \subset W(m)$.
Therefore, for $\bl\{W( m)\br\}^{\infty}_{m=1} \subset [X, Y]$ there is no function $N \colon (0, \infty) \rightarrow \mathbb N$ such that for any $m \in \mathbb N$ the inequality $\operatorname{cov}\bl(W(m), \varepsilon\br) \leq N(\varepsilon)$ is valid. So, according to Proposition \ref{thmG}, the segment [X, Y] is not precompact and, hence, a compact set.
\end{proof}



\end{document}